\documentclass[12pt]{article}
\usepackage{amsmath}
\usepackage{amssymb}
\usepackage{amsthm}
\usepackage{mathtools, bbm}
\usepackage[usenames,dvipsnames]{color}
\usepackage{setspace}
\usepackage{enumitem}
\usepackage{subcaption}
\usepackage{bm}
\usepackage[margin=2.54cm]{geometry}
\DeclareMathOperator{\st}{s.t.}
\DeclareMathOperator{\diag}{diag}
\DeclareMathOperator{\Diag}{Diag}

\DeclareMathOperator{\AR}{RLP}
\DeclareMathOperator{\AP}{Aff}
\DeclareMathOperator{\LB}{LB}
\newtheorem{theorem}{Theorem}
\newtheorem{lemma}{Lemma}

\newtheorem{proposition}{Proposition}
\newtheorem{assumption}{Assumption}

\newcommand{\nin}{\noindent}
\newcommand{\U}{{\mathcal{U}}}
\newcommand{\W}{{\mathcal{W}}}
\newcommand{\K}{\widehat{\mathcal{U}}}
\newcommand{\CP}{{\mathrm{CPP}}}
\newcommand{\COP}{{\mathrm{COP}}}
\newcommand{\Null}{\mathrm{Null}}
\newcommand{\X}{{\mathcal{X}}}

\newcommand{\M}{\mathrm{IA}}
\newcommand{\IB}{\mathrm{IB}}
\newcommand{\p}{\pi}
\renewcommand{\P}{\Pi}
\newcommand{\SYM}{{\mathcal{S}}}
\def\RR{{\mathbb{R}}}

\newcommand{\Q}{{\K \times \RR^m_+}}

\title{A Copositive Approach for \\ Two-Stage Adjustable Robust
Optimization \\ with Uncertain Right-Hand Sides}

\author{%
Guanglin Xu\thanks{Department of Management Sciences, University of Iowa,
Iowa City, IA, 52242-1994, USA. Email: {\tt guanglin-xu@uiowa.edu}.}%
\ \ \ \ \ \ \
Samuel Burer\thanks{Department of Management Sciences, University of Iowa,
Iowa City, IA, 52242-1994, USA. Email: {\tt samuel-burer@uiowa.edu}.}%
}

\date{September 23, 2016 \\ Revised: May 17, 2017}

\begin{document}

\maketitle

\begin{abstract}

\noindent We study two-stage adjustable robust linear programming
in which the right-hand sides are uncertain and belong to a convex,
compact uncertainty set. This problem is NP-hard, and the affine policy
is a popular, tractable approximation. We prove that under standard
and simple conditions, the two-stage problem can be reformulated as
a copositive optimization problem, which in turn leads to a class of
tractable, semidefinite-based approximations that are at least as
strong as the affine policy. We investigate several examples from the
literature demonstrating that our tractable approximations significantly
improve the affine policy. In particular, our approach solves exactly in
polynomial time a class of instances of increasing size for which the
affine policy admits an arbitrarily large gap.

\mbox{}


\end{abstract}

\begin{onehalfspace}

\section{Introduction} 

Ben-Tal et.~al.~\cite{Ben-Tal.Goryashko.Guslitzer.Nemirovski.2004}
introduced two-stage {\em adjustable robust optimization
(ARO)\/}, which considers both first-stage (``here-and-now'')
and second-stage (``wait-and-see'') variables. ARO can be
significantly less conservative than regular robust optimization,
and real-world applications of ARO abound: unit commitment in
renewable energy \cite{Bertsimas.Litvinov.Sun.Zhao.Zheng.2013,
Wang.Watson.Guan.2013, Zhao.Zeng.2012}, facility location problems
\cite{Ardestani-Jaafari.Delage.2014, Atamturk.Zhang.2007,
Gabrel.Lacroix.Murat.Remli.2014}, emergency supply chain planning
\cite{Ben-Tal.Chung.Mandala.Yao.2011}, and inventory management
\cite{Ben-Tal.Boaz.Shimrit.2009, Solyali.Cordeau.Laporte.2016}; see
also \cite{Ben-Tal.Golany.Nemirovski.Vial.2005, Fonseca.Rustem.2012,
Poss.Raack.2013}. We refer the reader to the excellent, recent tutorial
\cite{Delage.Iancu.2015} for background on ARO.

Since ARO is intractable in general
\cite{Ben-Tal.Goryashko.Guslitzer.Nemirovski.2004}, multiple
tractable approximations have been proposed for it. In certain
situations, a static, robust-optimization-based solution can be used
to approximate ARO, and sometimes this static solution is optimal
\cite{Ben-Tal.Nemirovski.1999,Bertsimas.Goyal.Lu.2014}. The {\em
affine policy\/} \cite{Ben-Tal.Goryashko.Guslitzer.Nemirovski.2004},
which forces the second-stage variables to be an affine function of
the uncertainty parameters, is another common approximation for ARO,
but it is generally suboptimal. Several nonlinear policies have also
been used to approximate ARO. Chen and Zhang \cite{Chen.Zhang.2009}
proposed the {\em extended affine policy\/} in which the primitive
uncertainty set is reparameterized by introducing auxiliary variables
after which the regular affine policy is applied. Bertsimas
et.~al.~\cite{Bertsimas.Iancu.Parrilo.2011} introduced a more accurate,
yet more complicated, approximation which forces the second-stage
variables to depend polynomially (with a user-specified, fixed degree)
on the uncertain parameters. Their approach yields a hierarchy
of Lasserre-type semidefinite approximations and can be extended
to multi-stage robust optimization. Ardestani-Jaafari and Delage
\cite{Ardestani-Jaafari.Delage.2016a} studied a robust optimization
problem featuring sums of piecewise linear functions, which is in
fact a special case of ARO, and they proposed approximations based on
mixed-integer linear programming and semidefinite programming.

The approaches just described provide upper bounds when ARO is
stated as a minimization. On the other hand, a single lower bound
can be calculated, for example, by fixing a specific value in the
uncertainty set and solving the resulting LP (linear program), and Monte Carlo
simulation over the uncertainty set can then be used to compute a
best lower bound. Finally, global approaches for solving ARO exactly
include column and constraint generation \cite{Zeng.Zhao.2013} and
Benders decomposition \cite{Bertsimas.Litvinov.Sun.Zhao.Zheng.2013,
Doulabi.Jaillet.Pesant.Rousseau.2016}.

In this paper, we consider the following two-stage adjustable robust
linear minimization problem with uncertain right-hand side:
\begin{equation} \tag{$RLP$} \label{equ:ar}
\begin{array}{lll}
v_{\AR}^* := & \min \limits_{x,y(\cdot)} & c^Tx + \max \limits_{u \in \U} d^Ty(u) \\
         & \st & Ax + By(u) \ge Fu \ \ \ \forall \; u \in \U \\
         &     & x \in \X,
\end{array}
\end{equation}
where $A \in \RR^{m \times n_1}, B \in \RR^{m \times n_2}, c \in
\RR^{n_1}$, $d \in \RR^{n_2}$, $F \in \RR^{m \times k}$ and $\X
\subseteq \RR^{n_1}$ is a closed convex set containing the first-stage
decision $x$. The uncertainty set $\U \subseteq \RR^k$ is compact,
convex, and nonempty, and in particular we model it as a slice of
a closed, convex, full-dimensional cone $\K \subseteq \RR_+\times
\RR^{k-1}$:
\begin{equation} \label{equ:UK}
\U := \{u \in \K : e_1^Tu = u_1 = 1\},
\end{equation}
where $e_1$ is the first canonical basic vector in $\RR^k$. In words,
$\K$ is the homogenization of $\U$. We choose this homogenized version
for notational convenience and note that it allows the modeling of
affine effects of the uncertain parameters. The second-stage variable is
$y(\cdot)$, formally defined as a mapping $y : \U \to \RR^{n_2}$. It is
well known that ($\AR$) is equivalent to
\begin{equation} \label{equ:ar'}
\begin{array}{lll}
v_{\AR}^* = & \min \limits_{x \in \X} & c^Tx + \max \limits_{u \in \U} \min
\limits_{y(u) \in \RR^{n_2}}
\{ d^T y(u) :  By(u) \ge Fu - Ax \},
\end{array}
\end{equation}
where $y(u)$ is a vector variable specifying the value of $y(\cdot)$ at
$u$.

Regarding ($\AR$), we make three standard assumptions.

\begin{assumption} \label{ass:tractable}
The closed, convex set $\X$ is computationally tractable, and the
closed, convex cone $\K$ is full-dimensional and computationally
tractable.
\end{assumption}

\noindent For example, $\X$ and $\K$ could be represented using a
polynomial number of linear, second-order-cone, and semidefinite
inequalities, each of which possesses a polynomial-time separation
oracle \cite{Grotschel.Lovasz.Shrijver.1981}.

\begin{assumption} \label{ass:feas}
Problem $(\AR)$ is feasible, i.e., there exists a choice $x \in \X$ and
$y(\cdot)$ such that $Ax + B y(u) \ge Fu$ for all $u \in \U$.
\end{assumption}

\noindent The existence of an affine policy, which can be checked
in polynomial time, is sufficient to establish that Assumption
\ref{ass:feas} holds.

\begin{assumption} \label{ass:boundedval}
Problem $(\AR)$ is bounded, i.e., $v^*_{\AR}$ is finite. 
\end{assumption}

\noindent Note that the negative directions of recession $\{ \tau : d^T
\tau < 0, B \tau \ge 0 \}$ for the innermost LP in (\ref{equ:ar'}) do
not depend on $x$ and $u$. Hence, in light of Assumptions \ref{ass:feas}
and \ref{ass:boundedval}, there must exist no negative directions of
recession; otherwise, $v^*_{\AR}$ would clearly equal $-\infty$. So
every innermost LP in (\ref{equ:ar'}) is either feasible with bounded value or
infeasible. In particular, Assumption \ref{ass:feas} implies that at
least one such LP is feasible with bounded value. It follows that the specific
associated dual LP $\max\{ (Fu - Ax)^T w : B^T w = d, w \ge 0 \}$ is
also feasible with bounded value. In particular, the fixed set
\[
    \W := \{ w \ge 0 : B^T w = d \}
\]
is nonempty. For this paper, we also make one additional assumption:
\begin{assumption} \label{ass:rcr}
Problem $(\AR)$ possesses relatively complete recourse, i.e., for all $x
\in \X$ and $u \in \U$, the innermost LP in (\ref{equ:ar'}) is feasible.
\end{assumption}

\noindent By the above discussion, Assumption \ref{ass:rcr} guarantees
that the innermost LP is feasible with bounded value, and hence every
dual $\max\{ (Fu - Ax)^T w : B^T w = d, w \ge 0 \}$ attains its optimal
value at an extreme point of $\W$.

In Section \ref{sec:copositive}, under Assumptions
\ref{ass:tractable}--\ref{ass:rcr}, we reformulate $(\AR)$ as an
equivalent copositive program, which first and foremost enables a
new perspective on two-stage robust optimization. Compared to most
existing copositive approaches for difficult problems, ours exploits
copositive duality; indeed, Assumption \ref{ass:rcr} is sufficient for
establishing strong duality between the copositive primal and dual. In
Section \ref{sec:tractableaffine}, we then apply a similar approach
to derive a new formulation of the affine policy, which is then, in
Section \ref{sec:results}, directly related to the copositive version
of $(\AR)$. This establishes two extremes: on the one side is the
copositive representation of ($\AR$), while on the other is the affine
policy. Section \ref{sec:results} also proposes semidefinite-based
approximations of $(\AR)$ that interpolate between the full copositive
program and the affine policy. Finally, in Section \ref{sec:examples},
we investigate several examples from the literature that demonstrate our
bounds can significantly improve the affine-policy value. In particular,
we prove that our semidefinite approach solves a class of instances of
increasing size for which the affine policy admits arbitrarily large
gaps. We end the paper with a short discussion of future directions in
Section \ref{sec:future}.

It is important to note that, even if Assumption \ref{ass:rcr} does
not hold, our copositive program still yields a valid upper bound on
$v^*_{\AR}$ that is at least as strong as the affine policy. More
comments are provided at the end of Section \ref{sec:copositive}; see
also Section \ref{sec:tractableaffine}.

We mention two studies that are closely related to ours. Chang
et~al.~\cite{Chang.Rao.Tawarmalani.2017} consider a particular
application of two-stage ARO in network design under uncertain
demands and uncertain path failures; their primary problem does not
contain explicit first-stage variables (although they do consider
an extension which does). The authors use LP duality to reformulate
their problem as a bilinear programming problem and subsequently
approximate it via the standard, LP-based reformulation-linearization
technique (RLT). They also show that their approximation improves
the affine policy. In a similar vein, Ardestani-Jaafari and
Delage~\cite{Ardestani-Jaafari.Delage.2016} introduce an approach for
($\AR$) that applies LP duality, RLT-style and semidefinite valid
inequalities, and semidefinite duality to obtain an approximation
of ($\AR$). In comparison to \cite{Chang.Rao.Tawarmalani.2017} and
\cite{Ardestani-Jaafari.Delage.2014}, we use copositive duality to
reformulate ($\AR$) exactly and then approximate it using semidefinite
programming. Although all three approaches are closely related, we
prefer our approach because it clearly separates the use of conic
duality from the choice of approximation. We also feel that our
derivation is relatively compact. In addition, both our paper and
\cite{Ardestani-Jaafari.Delage.2016} consider a general uncertainty
set but \cite{Ardestani-Jaafari.Delage.2016} focuses on a polyhedral
$\U$ from a practical point of view whereas our approach focuses on
the class of uncertainty sets that can be represented, say, by linear,
second-order-cone, and semidefinite inequalities.


On the same day (September 23, 2016) as the original version of this
article was posted on the online archive sites {\em Optimization
Online\/} and {\em arXiv\/}, the paper \cite{Hanasusanto.Kuhn.2016}
by Hanasusanto and Kuhn was also posted for the first time on
{\em Optimization Online\/}. It turns out that Corollary 1
of \cite{Hanasusanto.Kuhn.2016} is equivalent to our Theorem
\ref{thm:equivalence}, and so we mention it here for the reader's
reference. However, the copositive representations in the two papers
appear quite different due to notational choices, e.g., we use
homogenization and a general cone, while \cite{Hanasusanto.Kuhn.2016}
does not homogenize and focuses on polyhedral cones. In other aspects,
the two papers are quite different, e.g., our paper connects the
copositive representation with the affine policy, and we present a class
of examples that are solved exactly by our semidefinite approximation.

\subsection{Notation, terminology, and background} 

Let $\RR^n$ denote $n$-dimensional Euclidean space represented as
column vectors, and let $\RR_+^n$ denote the nonnegative orthant in
$\RR^n$. For a scalar $p \ge 1$, the $p$-norm of $v \in \RR^n$ is
defined $\|v\|_p := (\sum_{i=1}^n |v_i|^p)^{1/p}$, e.g., $\|v\|_1 =
\sum_{i=1}^n |v_i|$. We will drop the subscript for the $2$-norm, i.e.,
$\|v\| := \|v\|_2$. For $v,w \in \RR^n$, the inner product of $v$ and
$w$ is $v^Tw := \sum_{i=1}^n v_iw_i$. The symbol $\mathbbm{1}_n$ denotes
the all-ones vector in $\RR^n$.

The space $\RR^{m \times n}$ denotes the set of real $m \times n$
matrices, and the trace inner product of two matrices $A, B \in \RR^{m
\times n}$ is $A\bullet B := \text{trace}(A^TB)$. $\SYM^n$ denotes
the space of $n \times n$ symmetric matrices, and for $X \in \SYM^n$,
$X \succeq 0$ means that $X$ is positive semidefinite. In addition,
$\diag(X)$ denotes the vector containing the diagonal entries of
$X$, and $\Diag(v)$ is the diagonal matrix with vector $v$ along its
diagonal. We denote the null space of a matrix $A$ as $\Null(A)$, i.e.,
$\Null(A) := \{x : Ax =0\}$. For ${\cal K} \subseteq \RR^n$ a closed,
convex cone, ${\cal K}^*$ denotes its dual cone. For a matrix $A$ with
$n$ columns, the inclusion $\text{Rows}(A) \in {\cal K}$ indicates that
the rows of $A$---considered as column vectors---are members of ${\cal
K}$.

We next introduce some basics of {\em copositive programming} with
respect to the cone ${\cal K} \subseteq \RR^n$. The {\em
copositive cone\/} is defined as
\[
    \COP({\cal K}) := \{ M \in \SYM^n: x^T M x \ge 0 \ \forall \ x \in {\cal K} \},
\]
and its dual cone, the {\em completely positive cone\/}, is
\[
    \CP({\cal K}) := \{ X \in \SYM^n: X = \textstyle{\sum_i} x^i (x^i)^T, \ x^i \in {\cal K} \},
\]
where the summation over $i$ is finite but its cardinality is
unspecified. The term {\em copositive programming\/} refers to linear
optimization over $\COP({\cal K})$ or, via duality, linear optimization
over $\CP({\cal K})$. In fact, these problems are sometimes called
{\em generalized copositive programming\/} or {\em set-semidefinite
optimization\/} \cite{Burer.Dong.2012,Eichfelder.Jahn.2008} in contrast
with the standard case ${\cal K} = \RR_+^n$. In this paper, we work with
generalized copositive programming, although we use the shorter phrase
for convenience.

Finally, for the specific dimensions $k$ and $m$ of problem ($\AR$),
we let $e_i$ denote the $i$-th standard basis vector in $\RR^k$, and
similarly, $f_j$ denotes the $j$-th standard basis vector in $\RR^m$. We
will also use $g_1 := {e_1 \choose 0} \in \RR^{k+m}$.

\section{A Copositive Reformulation} \label{sec:copositive}

In this section, we construct a copositive representation of ($\AR$)
under Assumptions \ref{ass:tractable}--\ref{ass:rcr} by first
reformulating the inner maximization of (\ref{equ:ar'}) as a copositive
problem and then employing copositive duality.

Within (\ref{equ:ar'}), define 
\[
    \p(x) := \max_{u \in \U} \min_{y(u) \in \RR^{n_2}} \{ d^T y(u) : By(u) \ge Fu - Ax \}.
\]

\noindent The dual of the inner minimization is $\max_{w \in \W} (Fu -
Ax)^T w$, which is feasible as discussed in the Introduction. Hence,
strong duality for LP implies
\begin{equation} \label{equ:Q(x)}
    \p(x) = \max_{u \in \U} \max_w \{ (Fu - Ax)^T w : w \in \W \}
          = \max_{(u,w) \in \U \times \W} \ (Fu - Ax)^T w,
\end{equation}
In words, $\p(x)$ equals the optimal value of a bilinear program over
convex constraints, which is NP-hard in general \cite{Konno.1976}.

It holds also that $\p(x)$ equals the optimal value of an associated
copositive program (see \cite{Burer.2009,Burer.2011} for example), which
we now describe. Define
\begin{equation} \label{equ:zE}
    z := {u \choose w} \in \RR^{k + m}, \quad
    E := \begin{pmatrix} -de_1^T & B^T \end{pmatrix} \in \RR^{n_2 \times (k + m)},
\end{equation}
where $e_1 \in \RR^k$ is the first coordinate vector, and homogenize
via the relationship (\ref{equ:UK}) and the definition of $\W$:
\begin{align*}
\p(x) = \max \; & \ (F - Axe_1^T) \bullet wu^T  \\  
    \st \; \; & \ E z = 0  \\ 
    & \ z \in \Q, \quad g_1^Tz = 1,
\end{align*}
where $g_1$ is the first coordinate vector in $\RR^{k+m}$. The
copositive representation is thus
\begin{align} \label{equ:qxcppstandard}
    \p(x) = \max \; & (F - Axe_1^T) \bullet Z_{21} \\
\st \; \; &\diag(EZE^T) = 0 \nonumber \\
     & Z \in \CP(\Q), \quad g_1 g_1^T \bullet Z = 1, \nonumber
\end{align}
where $Z$ has the block structure
\[
    Z = \begin{pmatrix}
        Z_{11} & Z_{21}^T \\
        Z_{21} & Z_{21}
    \end{pmatrix} \in {\cal S}^{k+m}.
\]
Note that under positive semidefiniteness, which is implied by the
completely positive constraint, the constraint $\diag(EZE^T) = 0$ is
equivalent to $ZE^T = 0$; see proposition 1 of \cite{Burer.2011}, for
example. For the majority of this paper, we will focus on this second
version:
\begin{align} \label{equ:qxcpp}
    \p(x) = \max \; & (F - Axe_1^T) \bullet Z_{21} \\
\st \; \; & ZE^T = 0 \nonumber \\
     & Z \in \CP(\Q), \quad g_1 g_1^T \bullet Z = 1. \nonumber
\end{align}

By standard theory \cite[corollary 3.2d]{Schrijver.1986}, the extreme
points of $\W$ are contained in a ball $w^T w \le r_w$, where $r_w
> 0$ is a radius that is polynomially computable and representable
in the encoding length of the entries of $B$ and $d$ (assuming those
entries are rational). Hence, Assumption \ref{ass:rcr} guarantees that
the optimal value of $\max\{ (Fu - Ax)^T w : B^T w = d, w \ge 0 \}$
does not change when $w^T w \le r_w$ is enforced. In addition, because
$\U$ is bounded by Assumption \ref{ass:tractable}, there exists a
sufficiently large scalar $r_z > 0$ such that the constraint $z^T z
\le r_z$ is redundant. It follows from these observations that, in the
preceding argument, we can enforce $z^T z = u^T u + w^T w \le r := r_z +
r_w$ without cutting off all optimal solutions of (\ref{equ:Q(x)}).
Thus, the lifted and linearized constraint $I \bullet Z \le r$ can be
added to (\ref{equ:qxcpp}) without changing its optimal value, although
some feasible directions of recession may be cut off. We arrive at
\begin{align} \label{equ:qxcpp'}
    \p(x) = \max \; & (F - Axe_1^T) \bullet Z_{21} \\
\st \; \; & ZE^T = 0, \quad I \bullet Z \le r \nonumber \\
     & Z \in \CP(\Q), \quad g_1 g_1^T \bullet Z = 1. \nonumber
\end{align}
We remark that the procedure of bounding the vertices of $\W$
is similar in spirit to the scheme proposed in Proposition 6 of
\cite{Ardestani-Jaafari.Delage.2016}. 

Letting $\Lambda \in \RR^{(k + m) \times n_2}$, $\lambda \in \RR$,
and $\rho \in \RR$ be the respective dual multipliers of $ZE^T = 0$,
$g_1 g_1^T \bullet Z = 1$, and $I \bullet Z \le r$, standard
conic duality theory implies the dual of (\ref{equ:qxcpp'}) is
\begin{equation} \label{equ:cop_qx}
\begin{array}{ll}
\min \limits_{\lambda, \Lambda, \rho} & \lambda + r \rho \\
\st &  \lambda g_1 g_1^T - \tfrac12 G(x) + \tfrac12(E^T \Lambda^T + \Lambda E)
+ \rho I \in \COP(\Q) \\
& \rho \ge 0
\end{array}
\end{equation}
where
\[
    G(x) :=
    \begin{pmatrix}
        0 & (F - Axe_1^T)^T \\
             F - Axe_1^T & 0 
    \end{pmatrix}
    \in {\cal S}^{k + m}
\]
is affine in $x$. Holding all other dual variables fixed, for $\rho
> 0$ large, the matrix variable in (\ref{equ:cop_qx}) is strictly
copositive---in fact, positive definite---which establishes that
Slater's condition is satisfied, thus ensuring strong duality:

\begin{proposition} 
Under Assumption \ref{ass:rcr}, suppose $r > 0$ is a constant such that
$z^T z \le r$ is satisfied by all $u \in \U$ and all extreme points $w
\in \W$, where $z = (u,w)$. Then the optimal value of (\ref{equ:cop_qx})
equals $\p(x)$.
\end{proposition}

Now, with $\p(x)$ expressed as a minimization that depends affinely on
$x$, we can collapse (\ref{equ:ar'}) into a single minimization that
is equivalent to $(\AR)$:
\begin{equation} \tag{$\overline{RLP}$} \label{equ:cop_ar} 
\begin{array}{ll}
\min \limits_{x, \lambda, \Lambda, \rho} & c^T x + \lambda + r \rho \\
\st & x \in \X, \
\lambda g_1 g_1^T - \tfrac12 G(x) + \tfrac12(E^T \Lambda^T + \Lambda E) + \rho I
\in \COP(\Q) \\ & \rho \ge 0.
\end{array}
\end{equation}

\begin{theorem} \label{thm:equivalence}
The optimal value of ($\overline{RLP}$) equals $v^*_{\AR}$.
\end{theorem}

An equivalent version of (\ref{equ:cop_ar}) can be derived
based on the representation of $\p(x)$ in
(\ref{equ:qxcppstandard}):
\begin{equation} \label{equ:cop_ar'} 
\begin{array}{ll}
\min \limits_{x, \lambda, v, \rho} & c^T x + \lambda + r \rho \\
\st & x \in \X, \
\lambda g_1 g_1^T - \tfrac12 G(x) + E^T\Diag(v)E + \rho I
\in \COP(\Q) \\ & \rho \ge 0.
\end{array}
\end{equation}
Our example in Section \ref{ssec:example} will be based
on this version.

We remark that, even if Assumption \ref{ass:rcr} fails and strong
duality between (\ref{equ:qxcpp'}) and (\ref{equ:cop_qx}) cannot be
established, it still holds that the optimal value of ($\overline{RLP}$)
is an upper bound on $v^*_{\AR}$. Note that, in this case,
(\ref{equ:qxcpp'}) should be modified to exclude $I \bullet Z \le r$,
and $\rho$ should be set to 0 in (\ref{equ:cop_qx}).

\section{The Affine Policy} \label{sec:tractableaffine}

Under the affine policy, the second-stage decision variable $y(\cdot)$
in $(\AR)$ is modeled as a linear function of $u$ via a free variable $Y
\in \RR^{n_2 \times k}$:

\begin{equation} \tag{{\em Aff\/}} \label{equ:ap} 
\begin{array}{lll}
v_{\AP}^* := & \min \limits_{x, y(\cdot), Y} & c^Tx + \max \limits_{u \in \U} d^Ty(u) \\
                    & \st    & Ax + By(u) \ge Fu \ \ \ \forall \; u \in \U \\
                    &        & y(u) = Yu \quad\quad\quad\quad \ \forall \; u \in \U \\
                    &        & x \in \X. 
\end{array}
\end{equation}

\noindent Here, $Y$ acts as a ``dummy'' first-stage decision, and so
$(\AP)$ can be recast as a regular robust optimization
problem over $\U$. Specifically, using standard techniques
\cite{Ben-Tal.Goryashko.Guslitzer.Nemirovski.2004}, $(\AP)$ is
equivalent to
\begin{equation} \label{equ:conic_ap}
\begin{array}{ll}
\min \limits_{x,Y,\lambda}  & c^Tx + \lambda \\
 \st & \lambda e_1 - Y^Td \in {\K}^*   \\
      & \text{Rows}(Axe_1^T - F + BY) \in {\K}^* \\
      & x \in \X.
\end{array}
\end{equation}
Problem (\ref{equ:conic_ap}) is tractable, but in general, the affine
policy is only an approximation of $(\AR)$, i.e., $v_{\AR}^* <
v_{\AP}^*$. In what follows, we provide a copositive
representation ($\overline{\AP}$) of (\ref{equ:ap}), which is then used
to develop an alternative formulation $(\M)$ of (\ref{equ:conic_ap}).
Later, in Section \ref{sec:results}, problem $(\M)$ will be compared
directly to $(\overline{\AR})$.


Following the approach of Section \ref{sec:copositive}, we may express
$(\AP)$ as $\min \limits_{x \in \X, Y} c^Tx + \P(x,Y)$ where
\[
    \P(x,Y) := \max_{u \in \U} \min_{y(u) \in \RR^{n_2}}
    \{ d^T y(u) : By \ge Fu - Ax, \ y(u) = Yu \}.
\]
Note that we do not replace $y(u)$ everywhere by $Yu$ in the definition
of $\P(x,Y)$; this is a small but critical detail in the subsequent
derivations. The inner minimization has dual
\begin{align*}
    &\max \limits_{w \ge 0, v} \{ (Fu - Ax)^T w + (Yu)^T v : B^T w + v = d \} \\
    &= \max \limits_{w \ge 0} \left( (Fu - Ax)^T w + (Yu)^T (d - B^T w) \right).
\end{align*}
After collecting terms, homogenizing, and converting to copostive
optimization, we have
\begin{equation} \label{equ:cp_qxy}
\begin{array}{ll}
    \P(x,Y) = &\max \quad \frac12 (G(x) - H(Y)) \bullet Z \\
&\st \quad \, Z \in \CP(\Q), \quad g_1 g_1^T \bullet Z = 1
\end{array}
\end{equation}
with dual
\begin{equation} \label{equ:cop_qxy} 
\begin{array}{ll}
\min \limits_{\lambda} & \lambda \\
\st &  \lambda g_1 g_1^T - \tfrac12 G(x) + \tfrac12 H(Y)
\in \COP(\Q),
\end{array}
\end{equation}
where $G(x)$ is defined as in Section \ref{sec:copositive} and
\[
    H(Y) := \begin{pmatrix}
        - e_1 d^T Y - Y^T d e_1^T & (BY)^T \\
        BY & 0
    \end{pmatrix}
    \in {\cal S}^{k + m}.
\]
Since $\K$ has interior by Assumption \ref{ass:tractable}, it
follows that (\ref{equ:cp_qxy}) also has interior, and so Slater's
condition holds, implying strong duality between (\ref{equ:cp_qxy})
and (\ref{equ:cop_qxy}). Thus, repeating the logic of Section
\ref{sec:copositive}, $(\AP)$ is equivalent to

\begin{equation} \tag{$\overline{\text{\it Aff\/}}$} \label{equ:cop_ap}
\begin{array}{ll}
\min \limits_{x,\lambda, Y} & c^T x + \lambda \\
\st & x \in \X, \quad \lambda g_1 g_1^T - \tfrac12 G(x) + \tfrac12 H(Y)
\in \COP(\Q).
\end{array}
\end{equation}

\begin{proposition} 
    The optimal value of (\ref{equ:cop_ap}) is $v^*_{\AP}$. 
\end{proposition}


We now show that $\COP(\Q)$ in $(\overline{\AP})$ can
be replaced by a particular inner approximation without changing the
optimal value. Moreover, this inner approximation is tractable, so that
the resulting optimization problem serves as an alternative to the formulation
(\ref{equ:conic_ap}) of ($\AP$).

Using the mnemonic ``IA'' for ``inner approximation,'' we define 
\[
\M(\Q) := \left\{
    S = \begin{pmatrix} S_{11} & S_{21}^T \\ S_{21} & S_{22} \end{pmatrix} : 
\begin{array}{l}
    S_{11} = e_1\alpha^T + \alpha e_1^T, \  \alpha \in \K^*, \\
    \text{Rows}(S_{21}) \in \K^*, \  S_{22} \ge 0
\end{array}
\right\}.
\]
This set is tractable because it is defined by affine constraints in
$\K^*$ as well as nonnegativity constraints. Moreover, $\M(\Q)$ is
indeed a subset of $\COP(\Q)$:

\begin{lemma} 
$\M(\Q) \subseteq \COP(\Q)$.
\end{lemma} 

\begin{proof}
We first note that (\ref{equ:UK}) implies that the first
coordinate of every element of $\K$ is nonnegative; hence, $e_1 \in
\K^*$. Now, for arbitrary ${p \choose q} \in \Q$ and $S
\in \M(\Q)$, we prove $t := {p \choose q}^T S {p \choose
q} \ge 0$. We have
\[
t 
= {p \choose q}^T
\begin{pmatrix} S_{11} & S_{21}^T \\ S_{21} & S_{22} \end{pmatrix}
{p \choose q}
= p^TS_{11}p + 2 \, q^TS_{21}p + q^TS_{22}q.
\]
Analyzing each of the three summands separately, we first have
\[
e_1, \alpha \in \K^* \ \ \Longrightarrow \ \
p^TS_{11}p = p^T(e_1\alpha^T + \alpha e_1^T)p = 2(p^Te_1)(\alpha^Tp) \ge 0.
\]
Second, $p \in \K$ and $\text{Rows}(S_{21}) \in \K^*$ imply $S_{21} p
\ge 0$, which in turn implies $q^TS_{21}p = q^T(S_{21}p) \ge 0$ because
$q \ge 0$. Finally, it is clear that $q^TS_{22}q \ge 0$ as $S_{22} \ge
0$ and $q \ge 0$. Thus, $t \ge 0 + 0 + 0 = 0$, as desired.
\end{proof}

The following tightening of ($\overline{\AP}$) simply replaces $\COP(\Q)$
with its inner approximation $\M(\Q)$:
\begin{equation} \tag{$IA$} \label{equ:mcone_ap}
\begin{array}{lll}
v_{\M}^* \ := \ &\min \limits_{x,\lambda, Y} \quad &c^T x + \lambda \\
&\st &x \in \X, \quad \lambda g_1 g_1^T - \tfrac12 G(x) + \tfrac12 H(Y)
\in \M(\Q).
\end{array}
\end{equation}
\nin By construction, $v_{\M}^* \ge v_{\AP}^*$, but in fact these values
are equal.

\begin{theorem} \label{theorem:affine-tractable}
$v_{\M}^* = v_{\AP}^*$.
\end{theorem}

\begin{proof}
We show $v_{\M}^* \le v_{\AP}^*$ by demonstrating that every feasible
solution of (\ref{equ:conic_ap}) yields a feasible solution of ($\M$)
with the same objective value. Let $(x,Y,\lambda)$ be feasible for
(\ref{equ:conic_ap}); we prove 
\[
    S := \lambda g_1 g_1^T - \tfrac12 G(x) + \tfrac12 H(Y) \in \M(\Q),
\]
which suffices. Note that the block form of $S$ is
\[
    S = 
    \begin{pmatrix}
        \lambda e_1 e_1^T - \tfrac{1}{2}(e_1d^T Y + Y^T de_1^T) &
        \tfrac{1}{2}(Axe_1^T - F + BY)^T \\
        \tfrac{1}{2}(Axe_1^T - F + BY) & 0
    \end{pmatrix}.
\]
The argument decomposes into three
pieces. First, we define $\alpha := \tfrac{1}{2}(\lambda e_1 - Y^Td)$,
which satisfies $\alpha \in \K^*$ due to (\ref{equ:conic_ap}). Then
\begin{align*}
    S_{11}
    &= \lambda e_1 e_1^T - \tfrac{1}{2}(e_1 d^T Y + Y^T d e_1^T) \\
    &= \left( \tfrac12 \lambda e_1 e_1^T - \tfrac{1}{2} e_1 d^T Y \right) +
       \left( \tfrac12 \lambda e_1 e_1^T - \tfrac{1}{2}Y^T d e_1^T \right) \\
    &= e_1 \alpha^T + \alpha e_1^T
\end{align*}
as desired. Second, we have $2 \, \text{Rows}(S_{21}) =
\text{Rows}(Axe_1^T - F + BY) \in \K^*$ by (\ref{equ:conic_ap}).
Finally, $S_{22} = 0 \ge 0$.
\end{proof}

\section{Improving the Affine Policy} \label{sec:results}


A direct relationship holds between ($\overline{\AR}$) and $(\M)$:

\begin{proposition} \label{pro:tightening}
In problem $(\overline{\AR})$, write $\Lambda = {\Lambda_1 \choose
\Lambda_2}$, where $\Lambda_1 \in \RR^{k \times n_2}$ and $\Lambda_2
\in \RR^{m \times n_2}$. Problem $(\M)$ is a restriction of
(\ref{equ:cop_ar}) in which $\Lambda_2 = 0$, $Y$ is identified with
$\Lambda_1^T$, $\rho = 0$, and $\COP(\Q)$ is tightened to $\M(\Q)$.
\end{proposition}

\begin{proof}
Examining the similar structure of $(\overline{\AR})$ and $(\M)$,
it suffices to equate the terms $E^T \Lambda^T + \Lambda E$ and
$H(Y)$ in the respective problems under the stated restrictions. From
(\ref{equ:zE}),
\[
    E^T \Lambda^T + \Lambda E =
    \begin{pmatrix}
        -e_1 d^T \Lambda_1^T - \Lambda_1 d e_1^T & \Lambda_1 B^T - e_1 d^T \Lambda_2^T \\
        B \Lambda_1^T - \Lambda_2 d e_1^T & B \Lambda_2^T + \Lambda_2 B^T
    \end{pmatrix}.
\]
Setting $\Lambda_2 = 0$ and identifying $Y = \Lambda_1^T$, we see
\[
    E^T \Lambda^T + \Lambda E = 
    \begin{pmatrix}
        -e_1 d^T Y - Y^T d e_1^T & Y^T B^T \\
        B Y & 0
    \end{pmatrix} =
    H(Y),
\]
as desired.
\end{proof}

Now let $\IB(\Q)$ be any closed convex cone satisfying
\[
    \M(\Q) \subseteq
    \IB(\Q) \subseteq
    \COP(\Q),
\]
where the mnemonic ``IB'' stands for ``in between'', and consider the
following problem gotten by replacing $\COP(\Q)$ in
($\overline{\AR}$) with $\IB(\Q)$:
\begin{equation} \tag{$IB$} \label{equ:ccone_ar}
\begin{array}{lll}
v_{\IB}^* \ := \ &\min \limits_{x, \lambda, \Lambda} \quad & c^T x + \lambda \\
&\st & x \in \X, \quad \lambda g_1 g_1^T - \tfrac12 G(x) + \tfrac12(E^T \Lambda^T + \Lambda E) 
\in \IB(\Q).
\end{array}
\end{equation}
Problem $(\IB)$ is clearly a restriction of ($\overline{\AR}$), and
by Proposition \ref{pro:tightening}, it is simultaneously no tighter
than $(\M)$. Combining this with Theorems \ref{thm:equivalence} and
\ref{theorem:affine-tractable}, we thus have:

\begin{theorem} \label{theorem:approximaterlp}
$v_{\AR}^* \le v_{\IB}^* \le v_{\AP}^*$.
\end{theorem}

We end this section with a short discussion of example approximations
$\IB(\Q)$ for typical cases of $\K$. In fact, there are complete
hierarchies of approximations of $\COP(\Q)$ \cite{Zuluaga.et.al.2006},
but we present a relatively simple construction that starts from a given
inner approximation $\IB(\K)$ of $\COP(\K)$:

\begin{proposition} 
Suppose $\IB(\K) \subseteq \COP(\K)$, and define
\[
    \IB(\Q) :=
    \left\{
        S + M + R
        \ : \
        \begin{array}{c}
        S \in \M(\Q), \ M \succeq 0 \\
        R_{11} \in \IB(\K), \ R_{21} = 0, \ R_{22} = 0
        \end{array}
    \right\}.
\]
Then $\M(\Q) \subseteq \IB(\Q) \subseteq \COP(\Q)$.
\end{proposition}

\begin{proof}
For the first inclusion, simply take $M = 0$ and $R_{11} = 0$. For the
second inclusion, let arbitrary ${p \choose q} \in \Q$ be given. We need
to show
\[
    {p \choose q}^T
    \left(
        S + M + R
    \right)
    {p \choose q} = \textstyle{{p \choose q}}^T S {p \choose q} +
    \textstyle{{p \choose q}}^T M {p \choose q} + p^T R_{11} p \ge 0.
\]
The first term is nonnegative because $S \in \M(\Q)$; the second
term is nonnegative because $M \succeq 0$; and the third is nonnegative
because $R_{11} \in \COP(\K)$.
\end{proof}

When $\K = \{ u \in \RR^k : \|(u_2, \ldots, u_k)^T\| \le u_1 \}$ is the
second-order cone, it is known \cite{Sturm.Zhang.2003} that
\[
   \COP(\K) = \{ R_{11} = \tau J + M_{11} : \tau \ge 0, \ M_{11} \succeq 0 \},
\]
where $J = \Diag(1, -1, \ldots, -1)$. Because of this simple structure,
it often makes sense to take $\IB(\K) = \COP(\K)$ in practice. Note also
that $M_{11} \succeq 0$ can be absorbed into $M \succeq 0$ in the definition
of $\IB(\Q)$ above. When $\K
= \{ u \in \RR^k: Pu \ge 0 \}$ is a polyhedral cone based on some matrix
$P$, a typical inner approximation of $\COP(\K)$ is
\[
\IB(\K) := \{ R_{11} = P^T N P : N \ge 0 \},
\]
where $N$ is a symmetric matrix variable of appropriate size. This corresponds
to the RLT approach of \cite{Anstreicher.2009, Burer.2015, Sherali.Adams.2013}.

\section{Examples} \label{sec:examples}

In this section, we demonstrate our approximation $v_{\IB}^*$ satisfying
$v^*_{\AR} \le v_{\IB}^* \le v_{\AP}^*$ on several examples from the
literature. The first example is treated analytically, while the remaining
examples are verified numerically. All computations are conducted with
Mosek version 8.0.0.28 beta \cite{mosek} on an Intel Core i3 2.93 GHz Windows
computer with 4GB of RAM and implemented using the modeling language
YALMIP \cite{lofberg2004yalmip} in MATLAB (R2014a).

\subsection{A temporal network example} \label{ssec:example}

The paper \cite{Wiesemann.Kuhn.Rustem.2011} studies a so-called {\em
temporal network\/} application, which for any integer $s \ge 2$ leads
to the problem (\ref{equ:temporal_simple}) below. The uncertainty set is
$\Xi \subseteq \RR^s$; the first-stage decision $x$ is fixed, say, at 0;
and $y(\cdot)$ maps into $\RR^s$:
\begin{equation} \label{equ:temporal_simple}
\begin{array}{llll}
\min \limits_{y(\cdot)} \, & \max \limits_{\xi \in \Xi} \ y(\xi)_s \\
\st & y(\xi)_1 \ge \max\{\xi_1, 1 - \xi_1\}  & \forall \; \xi \in \Xi \\
    & y(\xi)_2 \ge \max\{\xi_2, 1 - \xi_2\} + y(\xi)_1 & \forall \; \xi \in \Xi \\
    &\vdots \\
    & y(\xi)_s \ge \max\{\xi_s, 1 - \xi_s\} + y(\xi)_{s-1} & \forall \; \xi \in \Xi.
\end{array}
\end{equation}
Note that each of the above linear constraints can be expressed
as two separate linear constraints. The authors of \cite{Wiesemann.Kuhn.Rustem.2011}
consider a polyhedral uncertainty set (based on the 1-norm). A related paper
\cite{Hadjiyiannis.Goulart.Kuhn.2011} considers a conic uncertainty set
(based on the 2-norm) for $s=2$; we will extend this to $s \ge 2$. In
particular, we consider the following two uncertainty sets for general
$s$:
\begin{align*}
    \Xi_1 &:= \{ \xi \in \RR^{s} : \|\xi - \tfrac12 \mathbbm{1}_s \|_1 \le \tfrac12 \}, \\
    \Xi_2 &:= \{ \xi \in \RR^{s} : \|\xi - \tfrac12 \mathbbm{1}_s \| \le \tfrac12 \},
\end{align*}
where $\mathbbm{1}_s$ denotes the all-ones vector in $\RR^s$. For $j =
1, 2$, let $v_{\AR, j}^*$ and $v^*_{\AP, j}$ be the robust and affine
values associated with (\ref{equ:temporal_simple}) for the uncertainty
set $\Xi_j$. Note that $\Xi_1 \subseteq \Xi_2$, and hence $v_{\AR, 1}^*
\le v_{\AR, 2}^*$. The papers \cite{Hadjiyiannis.Goulart.Kuhn.2011,
Wiesemann.Kuhn.Rustem.2011} show that $v_{\AP,1}^* = v_{\AP,2}^* =
s$, and \cite{Wiesemann.Kuhn.Rustem.2011} establishes $v_{\AR,1}^*
= \tfrac12(s + 1)$. Moreover, we prove the following result in the
Appendix:

\begin{lemma} \label{lemma:calculateVar2}
$v_{\AR, 2}^* = \tfrac12(\sqrt{s} + s)$. 
\end{lemma}

\noindent Overall, we see that each $j = 1,2$ yields a class of
problems with arbitrarily large gaps between the true robust adjustable
and affine-policy values.

Using the change of variables
\[
    u := (1, u_2, \ldots, u_{s+1})^T = (1, 2 \xi_1 - 1, \ldots, 2 \xi_s - 1)^T \in \RR^{s+1},
\]
for each $\Xi_j$, we may cast (\ref{equ:temporal_simple}) in the form
of ($\AR$) by setting $x = 0$, defining
\[
    m = 2s, \quad k = s + 1, \quad n_2 = s,
\]
and taking $\K_j$ to be the $k$-dimensional cone associated with the
$j$-norm. For convenience, we continue to use $s$
in the following discussion, but we will remind the reader of the
relationships between $s$, $m$, $k$, and $n_2$ as necessary (e.g., $s =
m/2$). We also set
\[
    d = (0,\ldots,0,1)^T \in \RR^s,
\]
\[
B = \begin{pmatrix}
       \phantom{-}1 & \phantom{-}0 & 0 & \cdots & \phantom{-}0 & 0 \\
       \phantom{-}1 & \phantom{-}0 & 0 & \cdots & \phantom{-}0 & 0 \\
      -1 & \phantom{-}1 & 0 & \cdots & \phantom{-}0 & 0 \\
      -1 & \phantom{-}1 & 0 & \cdots & \phantom{-}0 & 0 \\
       \phantom{-}0 & -1 & 1 & \cdots & \phantom{-}0 & 0 \\
       \phantom{-}0 & -1 & 1 & \cdots & \phantom{-}0 & 0 \\
\phantom{-}\vdots & \vdots & \vdots & \ddots & \phantom{-}\vdots & \vdots \\
       \phantom{-}0 & \phantom{-}0 & 0 & \cdots & \phantom{-}1 & 0 \\
       \phantom{-}0 & \phantom{-}0 & 0 & \cdots & \phantom{-}1 & 0 \\
       \phantom{-}0 & \phantom{-}0 & 0 & \cdots & -1 & 1 \\
       \phantom{-}0 & \phantom{-}0 & 0 & \cdots & -1 & 1 
\end{pmatrix} \in \RR^{2s \times s}, \ \ \
F = \frac12 \begin{pmatrix}
1 &  \phantom{-}1 &  \phantom{-}0 & \cdots &  \phantom{-}0 \\
1 & -1 &  \phantom{-}0 & \cdots &  \phantom{-}0 \\
1 &  \phantom{-}0 &  \phantom{-}1 & \cdots &  \phantom{-}0 \\
1 &  \phantom{-}0 & -1 & \cdots &  \phantom{-}0 \\
\vdots &  \phantom{-}\vdots &  \phantom{-}\vdots & \ddots &  \phantom{-}\vdots \\
1 &  \phantom{-}0 &  \phantom{-}0 & \cdots &  \phantom{-}1 \\
1 &  \phantom{-}0 &  \phantom{-}0 & \cdots & -1 
\end{pmatrix} \in \RR^{2s \times (s+1)}.
\]
Furthermore,
\[
    \K_2 := \{ u \in \RR^{s+1} : \| (u_2,\ldots,u_{s+1})^T \| \le u_1 \}
\]
is the second-order cone, and
\[
    \K_1 := \{ u \in \RR^{s+1} : P u \ge 0 \},
\]
where each row of $P \in \RR^{2^{s} \times (s+1)}$ has the following form:
$(1, \pm 1, \ldots, \pm 1)$. That is, each row is an $(s+1)$-length
vector with a 1 in its first position and some combination of $+1$'s
and $-1$'s in the remaining $s$ positions. Note that the size of $P$ is
exponential in $s$. Using extra nonnegative variables, we could also
represent $\K_1$ as the projection of a cone with size
polynomial in $s$, and all of the subsequent discussion would
still apply. In other words, the exact representation of $\K_1$ is not
so relevant to our discussion here; we choose the representation
$Pu \ge 0$ in the original space of variables for convenience.

It is important to note that, besides $\K_1$ and $\K_2$, all
other data required for representing (\ref{equ:temporal_simple}) in the form of
($\AR$), such as the matrices $B$ and $F$, do not depend on
$j$. Assumptions \ref{ass:tractable}--\ref{ass:boundedval} clearly hold,
and the following proposition shows that (\ref{equ:temporal_simple})
also satisfies Assumption \ref{ass:rcr}:

\begin{proposition}
For (\ref{equ:temporal_simple}) and its formulation as an instance of
(\ref{equ:ar}), $\W$ is nonempty and bounded.
\end{proposition}

\begin{proof}
The system $B^Tw = d$ is equivalent to the $2s -1$ equations $w_1 +
w_2 = 1$, $w_2 + w_3 = 1, \cdots, w_{2s-1} + w_{2s} = 1$. It is thus
straightforward to check that $\W$ is nonempty and bounded.
\end{proof}

\subsubsection{The case $j=2$}

Let us focus on the case $j=2$; we continue to make use of the subscript
$2$. Recall $v^*_{\AR,2} = \tfrac12 (\sqrt{s} + s)$, and consider
problem ($\IB_2$) with $\IB(\K_2 \times \RR_+^{2s})$ built as described
for the second-order cone at the end of Section \ref{sec:results}.
We employ the equivalent formulation (\ref{equ:cop_ar'}) of
($\overline{\AR}$), setting $x = 0$ and replacing $\COP(\K_2
\times \RR_+^{2s})$ by $\IB(\K_2 \times \RR_+^{2s})$:
\begin{equation} \label{equ:IB2}
\begin{array}{lll}
v_{\IB,2}^* = & \min
& \lambda + r\rho\\
& \st & \lambda g_1g_1^T
- \tfrac12 G(0) + E^T\Diag(v)E + \rho I
\in \IB(\K_2 \times \RR_+^{2s}) \\ 
& & \rho \ge 0.
\end{array}
\end{equation}
Note that the dimension of $g_1$ is $k + m = (s+1) + 2s = 3s + 1$.

Substituting the definition of $\IB(\K_2
\times \RR_+^{2s})$ from Section \ref{sec:results}, using the fact that
$\K_2^* = \K_2$, and simplifying, we have
\begin{equation} \label{equ:temporal1}
\begin{array}{lll}
v_{\IB,2}^* = & \min 
& \lambda + r\rho\\
& \st &
\rho I + 
\lambda g_1g_1^T 
- \tfrac12 G(0) 
+ E^T\Diag(v)E
- S - R\succeq 0  \\
& & \rho \ge 0, \ S_{11} = e_1\alpha^T + \alpha e_1^T, \ \alpha \in \K_2, \ S_{22} \ge 0, \ \text{\rm Rows}(S_{21}) \in \K_2 \\
& & R_{11} = \tau J, \ \tau \ge 0, \ R_{21} = 0, \ R_{22} = 0.
\end{array}
\end{equation}


\begin{proposition} \label{prop:vib2Equalsvar2}
For any $\rho > 0$, (\ref{equ:temporal1})
has a feasible solution with objective value $v_{\AR,2}^* + r
\rho$. 
\end{proposition}

\begin{proof}
See the Appendix.
\end{proof}

\begin{theorem}
$v_{\IB,2}^* = v_{\AR,2}^*$
\end{theorem}

\begin{proof}
We know $v^*_{\AR,2} \le v^*_{\IB,2}$ by Theorem \ref{theorem:approximaterlp}. 
Moreover we have $v_{\IB,2}^* \le v_{\AR,2}^* + r \rho$ for any $\rho >0$
by Proposition \ref{prop:vib2Equalsvar2}. Thus, by letting $\rho \to 0$, we have 
$v_{\IB,2}^* \le v_{\AR,2}^*$, which completes the proof. 
\end{proof}

For completeness---and also to facilitate Section \ref{ssec:example2}
next---we construct the corresponding optimal solution of the dual of
(\ref{equ:IB2}), which can be derived from (\ref{equ:qxcppstandard})
by setting $x = 0$, adding the redundant constraint $I
\bullet Z \le r$, and replacing $\CP(\K_2 \times \RR_+^{2s})$ by its
relaxation $\IB(\K_2 \times \RR_+^{2s})^*$, the dual cone of $\IB(\K_2
\times \RR_+^{2s})$. Specifically, the dual is
\begin{equation} \label{equ:temporal1_dual}
\begin{array}{lll}
v_{\IB,2}^* = & \max & F \bullet Z_{21} \\
& \st & \diag(E Z E^T) = 0, \ I \bullet Z \le r  \\
& & J \bullet Z_{11} \ge 0, \ Z_{11}e_1 \in \K_2, \ Z_{22} \ge 0, \
     \text{\rm Rows}(Z_{21}) \in \K_2 \\
& & Z \succeq 0, \ g_1g_1^T \bullet Z = 1.
\end{array}
\end{equation}
In particular, we construct the optimal solution of (\ref{equ:temporal1_dual}) 
in the following proposition:
\begin{proposition} \label{prop:dualoptimal}
Define
\[
Z =
\frac14 \left[
{2e_1 \choose \mathbbm{1}_m} {2e_1 \choose \mathbbm{1}_m}^T
+
\sum_{i=1}^{s}
{\tfrac{2}{\sqrt{s}} e_{i+1} \choose f_{2i-1} - f_{2i} }
{\tfrac{2}{\sqrt{s}} e_{i+1} \choose f_{2i-1} - f_{2i} }^T
\right],
\]
where each $e_{\bullet}$ is a canonical basis vector in $\RR^k = \RR^{s+1}$,
each $f_{\bullet}$ is a canonical basis vector in $\RR^m = \RR^{2s}$, and
$\mathbbm{1}_m \in \RR^m$ is the all-ones vector. Then, $Z$
is the optimal solution of (\ref{equ:temporal1_dual}). 
\end{proposition}

\begin{proof}
See the Appendix. 
\end{proof}

\subsubsection{The case $j=1$} \label{ssec:example2}

Recall that $\Xi_1$ is properly contained in $\Xi_2$. So $v_{\AR,1}^*$
cannot exceed $v_{\AR,2}^*$ due to its smaller uncertainty set. In
fact, as discussed above, we have
$\tfrac12(\sqrt{s} + 1) = v_{\AR,1}^* < v^*_{\AR,2} = \tfrac12(\sqrt{s}
+ s)$ and $v^*_{\AP,1} = v^*_{\AP,2} = s$. In this subsection, we
further exploit the inclusion $\Xi_1 \subseteq \Xi_2$ and the results
of the previous subsection (case $j=2$) to prove that, for the
particular tightening $\IB(\K_1 \times \RR_+^{2s})$ proposed at the end
of Section \ref{sec:results}, we have $v_{\AR,1}^* < v_{\IB,1}^* =
\tfrac12(\sqrt{s} + s) < v^*_{\AP,1}$. In other words, the case $j=1$
provides an example in which our approach improves the affine value but
does not completely close the gap with the robust value.
Our main result of this case is given in the following proposition. 

\begin{proposition} \label{prop:case1}
$v_{\IB,1}^* = v_{\IB,2}^* = \tfrac12(\sqrt{s} + s)$.
\end{proposition}

\begin{proof}
See the Appendix. 
\end{proof}

\subsection{Multi-item newsvendor problem} \label{ssec:newsvendor}

In this example, we consider the same robust multi-item newsvendor
problem discussed in \cite{Ardestani-Jaafari.Delage.2016}:
\begin{equation} \label{equ:vendor}
\max \limits_{x \ge 0} \min \limits_{\xi \in \Xi} \sum_{j \in {\cal J}}
\big[
    r_j \min(x_j, \xi_j) -c_j x_j + s_j \max (x_j - \xi_j, 0) - p_j \max (\xi_j - x_j, 0)
\big],
\end{equation}
where ${\cal J}$ represents the set of products; $x$ is the vector of
nonnegative order quantities $x_j$ for all $j \in {\cal J}$; $\xi \in
\Xi$ is the vector of uncertain demands $\xi_j$ for all $j \in {\cal
J}$; $r_j, c_j, s_j$, and $p_j$ denote the sale price, order cost,
salvage price, and shortage cost of a unit of product $j$ with $s_j \le
\min(r_j, c_j)$. Problem (\ref{equ:vendor}) is equivalent to
\begin{equation} \label{equ:two-stage-vendor}
\begin{array}{lll}
\max \limits_{x, y(\cdot)} & \min \limits_{\xi \in \Xi} \sum_{j \in {\cal J}} y_j(\xi) & \\
\st & y_j(\xi) \le (r_j - c_j) x_j - (r_j - s_j)(x_j - \xi_j) & \forall \, j \in {\cal J}, \, \xi \in \Xi \\  
     & y_j(\xi) \le (r_j - c_j)x_j - p_j(\xi_j - x_j) &  \forall \, j \in {\cal J}, \, \xi \in \Xi \\
     & x \ge 0. 
\end{array}
\end{equation}
We consider the same instance as in \cite{Ardestani-Jaafari.Delage.2016}
with ${\cal J} = \{1,2,3\}$, 
\[
r = (80, 80, 80), \ c = (70, 50, 20), \ s = (20, 15, 10), \ p =(60, 60, 50),
\]
and
\[
\Xi := \left\{ \xi : \exists \, \zeta^+, \ \zeta^- \st
\begin{array}{l} 
\zeta^+ \ge 0, \ \zeta^- \ge 0 \\
 \zeta^+_j + \zeta^-_j \le 1  \ \forall \, j \in {\cal J} \\ 
 \sum_{j \in {\cal J}} (\zeta^+_j + \zeta^-_j) = 2 \\
 \xi_1 = 80 + 30(\zeta_1^+ + \zeta_2^+ - \zeta_1^- - \zeta_2^-) \\
 \xi_2 = 80 + 30(\zeta_2^+ + \zeta_3^+ - \zeta_2^- - \zeta_3^-) \\
 \xi_3 = 60 + 20(\zeta_3^+ + \zeta_1^+ - \zeta_3^- - \zeta_1^-) 
  \end{array} \right\}.
\]

Omitting the details, we reformulate problem
(\ref{equ:two-stage-vendor}) as an instance of (\ref{equ:ar}) in
minimization form. Assumption \ref{ass:tractable} clearly holds, and by
using a method called {\em enumeration of robust linear constraints}
in \cite{Gorissen.Hertog.2013}, we have $v_{\AR}^* = -825.83$ (so
Assumption \ref{ass:boundedval} holds). Moreover, the affine-policy
value is $v_{\AP}^*=-41.83$, and thus Assumption \ref{ass:feas} holds.
As mentioned at the end of Section \ref{sec:copositive}, whether or
not Assumption \ref{ass:rcr} holds, in practice our approach still
provides an upper bound. Indeed, we solve $(\IB)$ with the approximating
cone $\IB(\Q)$ defined in Section \ref{sec:results}, where $\K$ is a
polyhedral cone, and obtain $v_{\IB}^* = -411.08$, which closes the gap
significantly. The first-stage decisions given by the affine policy and
our approach, respectively, are
\[
\begin{array}{c}
x_{\AP}^* \approx (52.083, 104.400, 80.000), \ \ \ x_{\IB}^* \approx (57.118, 78.162, 77.473).
\end{array}
\]
For the same instance, the paper \cite{Ardestani-Jaafari.Delage.2016}
reports the same upper bound. Indeed, it appears that the
specification of our cone $\IB(\Q)$ corresponds directly to the
classes of valid inequalities that they include in their approach
\cite{Ardestani-Jaafari.Delage.2017}, but we have not proved this
formally.


\subsection{Lot-sizing problem on a network} \label{ssec:network}

We next consider a network lot-sizing problem derived from section 5 of
\cite{Bertsimas.Ruiter.2016} for which the mathematical formulation is:
\[
\begin{array}{lll}
\min \limits_{x, y(\cdot)} & c^Tx + \max \limits_{\xi \in \Xi}
    \sum_{i=1}^N \sum_{j=1}^N t_{ij}y(\xi)_{ij} & \\
\st & x_i + \sum_{j=1}^N y(\xi)_{ji} - \sum_{j=1}^Ny(\xi)_{ij} \ge \xi_i &
        \forall \ \xi \in \Xi, \ i=1, \ldots, N \\
    & y(\xi)_{ij} \ge 0 & \forall \ \xi \in \Xi, \ i, j = 1, \ldots, N \\
    & 0 \le x_i \le V_i & \forall \ i = 1, \ldots, N,
\end{array}
\]
where $N$ is the number of locations in the network, $x$ denotes the
first-stage stock allocations, $y(\xi)_{ij}$ denotes the second-stage
shipping amounts from location $i$ to location $j$, and the uncertainty
set is the ball $\Xi := \{\xi : \|\xi\| \le \Gamma \}$ for a given
radius $\Gamma$. (The paper \cite{Bertsimas.Ruiter.2016} uses a
polyhedral uncertainty set, which we will also discuss below.) The
vector $c$ consists of the first-stage costs, the $t_{ij}$ are the
second-stage transportation costs for all location pairs, and $V_i$
represents the capacity of store location $i$. We refer the reader to
\cite{Bertsimas.Ruiter.2016} for a full description.

Consistent with \cite{Bertsimas.Ruiter.2016}, we consider an instance
with $N=8$, $\Gamma = 10\sqrt{N}$, each $V_i=20$, and each $c_i = 20$.
We randomly generate the positions of the $N$ locations from $[0,
10]^2$ in the plane. Then we set $t_{ij}$ to be the (rounded) Euclidean
distances between all pairs of locations; see Table \ref{Tab:cost}. 

\begin{table}[tbp]
\centering
\begin{tabular}{c|cccccccc}  
& \multicolumn{8}{c}{Location $i$} \\
Location $j$ & 1 & 2 & 3 & 4 & 5 & 6 & 7 & 8   \\ \hline
1 & 0 & 4 & 3 & 2 & 2 & 2 & 3 & 5  \\ 
2 & 4 & 0 & 6 & 5 & 4 & 4 & 2 & 8  \\
3 & 3 & 6 & 0 & 1 & 5 & 2 & 6 & 2  \\ 
4 & 2 & 5 & 1 & 0 & 4 & 1 & 4 & 3 \\
5 & 2 & 4 & 5 & 4 & 0 & 4 & 2 & 7 \\ 
6 & 2 & 4 & 2 & 1 & 4 & 0 & 4 & 4  \\ 
7 & 3 & 2 & 6 & 4 & 2 & 4 & 0 & 7  \\ 
8 & 5 & 8 & 2 & 3 & 7 & 4 & 7 & 0 
\end{tabular}
\caption{Unit transportation costs $t_{ij}$ associated with pairs of locations}
\label{Tab:cost}
\end{table}

Omitting the details, we reformulate this problem as an instance of
(\ref{equ:ar}), and we calculate $v^*_{\text{LB}} = 1573.8$ (using
the Monte Carlo sampling procedure mentioned in the Introduction)
and $v_{\AP}^*=1950.8$. It is also easy to see that Assumption
\ref{ass:tractable} holds, and the existence of an affine policy
implies that Assumption \ref{ass:feas} holds. Moreover, Assumption
\ref{ass:boundedval} holds because the original objective value above is
clearly bounded below by 0. Again, as mentioned at the end of Section
\ref{sec:copositive}, whether or not Assumption \ref{ass:rcr} holds,
in practice we can still use our approach to calculate bounds. We solve
(\ref{equ:ccone_ar}) with the approximating cone $\IB(\Q)$ defined in
Section \ref{sec:results}, where $\K$ is the second-order cone, and
obtain $v_{\IB}^* = 1794.0$, which closes the gap significantly. The
first-stage allocations given by the affine policy and our approach,
respectively, are
\[
\begin{array}{c}
x_{\AP}^* \approx (9.097, 11.246, \, 9.516, 8.320, 10.384, 9.493, 10.211, 12.316), \\
x_{\IB}^* \approx (0.269, 16.447, 15.328, 0.091, 18.124, 0.375,  \,9.951, 19.934).
\end{array}
\]

Letting other data remain the same, we also ran tests on a budget
uncertainty set $\Xi := \{\xi : 0 \le \xi \le \hat \xi e, \, e^T\xi \le
\Gamma \}$, where $\hat \xi = 20$ and $\Gamma = 20\sqrt{N}$, which is
consistent with \cite{Bertsimas.Ruiter.2016}. We found that, in this
case, our method did not perform better than the affine policy.

\subsection{Randomly generated instances} \label{ssec:random_example}

Finally, we used the same method presented in
\cite{Hadjiyiannis.Goulart.Kuhn.2011} to generate random instances of
(\ref{equ:ar}) with $(k,m,n_1,n_2) = (17, 16, 3, 5)$, $\X = \RR^{n_1}$,
$\U$ equal to the unit ball, and $\K$ equal to the second-order cone.
Specifically, the instances are generated as follows: (i) the elements
of $A$ and $B$ are independently and uniformly sampled in $[-5,5]$; (ii)
the rows of $F$ are uniformly sampled in $[-5, 5]$ such that each row is
in $-\K^* = -\K$ guaranteeing $Fu \le 0$ for all $u \in \U$; and (iii) a
random vector $\mu \in \RR^m$ is repeatedly generated according to the
uniform distribution on $[0,1]^m$ until $c := A^T\mu \ge 0$ and $d :=
B^T\mu \ge 0$. Note that, by definition, $\mu \in \W$.

Clearly Assumption \ref{ass:tractable} is satisfied. In addition, we can
see that Assumption \ref{ass:feas} is true as follows. Consider $x = 0$
and set $y(\cdot)$ to be the zero map, i.e., $y(u) = 0$ for all $u \in
\U$. Then $Ax + B y(u) \ge Fu$ for all $u$ if and only $0 \ge Fu$ for
all $u$, which has been guaranteed by construction. Finally, Assumption
\ref{ass:boundedval} holds due to the following chain, where $\p(x)$ is
defined as at the beginning of Section \ref{sec:copositive}:
\begin{align*}
    c^T x + \p(x)
    &= c^T x + \max \limits_{u \in \U} \max \limits_{w \in \W} (Fu - Ax)^T w \\
    &\ge c^T x + \max \limits_{u \in \U} (Fu - Ax)^T \mu 
    = c^T x - (Ax)^T \mu + \max \limits_{u \in \U} (Fu)^T \mu \\
    &= (c - A^T \mu)^T x + \max \limits_{u \in \U} (Fu)^T \mu 
    = 0^T x + \max \limits_{u \in \U} (Fu)^T \mu \\
    &> -\infty.
\end{align*}
We do not know if Assumption \ref{ass:rcr} necessarily holds
for this construction, but as mentioned at the end of Section
\ref{sec:copositive}, our approximations still hold even if Assumption
\ref{ass:rcr} does not hold.

For 1,000 generated instances, we computed $v_{\AP}^*$, the lower bound
$v^*_{\text{LB}}$ from the sampling procedure of the Introduction, and
our bound $v_{\IB}^*$ using the the approximating cone $\IB(\Q)$ defined
in Section \ref{sec:results}, where $\K$ is the second-order cone. Of
all 1,000 instances, 971 have $v^*_{\text{LB}} < v_{\IB}^* = v_{\AP}^*$,
while the remaining 29 have $v^*_{\text{LB}} < v_{\IB}^* < v_{\AP}^*$.
For those 29 instances with a positive gap, the average relative gap
closed is 20.2\%, where
\[
\text{relative gap closed} :=
\frac{v_{\AP}^* - v_{\IB}^*}{v_{\AP}^* - v_{\LB}^*} \times 100\%.
\]

\subsection{Computational details}

Table \ref{Tab:statistics} illustrates some computational details of 
the three numerical examples in Sections \ref{ssec:newsvendor}--\ref{ssec:random_example}. 
The statistics on the sizes of the
conic programs are reported by Mosek. We list the number of scalar 
variables ({\em scalars\/}), the number of second-order cones
({\em cones\/}), the number of positive semidefinite matrices along 
with their size ({\em matrices (size)\/}), and the number of linear 
constraints ({\em constraints\/}) in Table \ref{Tab:statistics}. We also
report the computation time in the last column. Note that all the $1,000$
instances in Section \ref{ssec:random_example} have the same problem
size and the computation time is the average of all the instances. 

\begin{table}[tbp]
\centering
\begin{tabular}{lrrrrrr}
Example & scalars & cones & matrices (size) & constraints & time (sec) \\ \hline
Sec.~\ref{ssec:newsvendor} & 603 & 0 & 1 ($13 \times 13$) & 448 & 0.23  \\
Sec.~\ref{ssec:network} & 14861 & 65 & 1 ($73 \times 73$) & 12210 & 35.00 \\
Sec.~\ref{ssec:random_example} & 2741 & 17 & 1 ($33 \times 33$) & 1870 &  0.73
\end{tabular}
\caption{Illustration of the sizes of problems and the computation times for the examples as reported by Mosek}
\label{Tab:statistics}
\end{table}

\section{Future Directions} \label{sec:future}

In this paper, we have provided a new perspective on the two-stage problem
$(\AR)$. It would be interesting to study tighter inner approximations
$\IB(\Q)$ of $\COP(\Q)$ or to pursue other classes of problems, such
as the one described in Section \ref{ssec:example}, for which our
approach allows one to establish the tractability of $(\AR)$. A
significant open question for our approach---one which we have not been
able to resolve---is whether the copositive approach corresponds to
enforcing a particular class of policies $y(\cdot)$. For example, the
paper \cite{Bertsimas.Iancu.Parrilo.2011} solves $(\AR)$ by employing
polynomial policies, but the form of our ``copositive policies'' is
unclear even though we have proven they are rich enough to solve
$(\AR)$. A related question is how to extract a specific policy
$y(\cdot)$ from the solution of the approximation $(\IB)$.

\section*{Acknowledgments}

The authors would like to thank Qihang Lin for many helpful discussions
regarding the affine policy at the beginning of the project and Erick
Delage and Amir Ardestani-Jaafari for thoughtful discussions, for
relaying the specific parameters of the instance presented in Section
\ref{ssec:newsvendor}, and for pointing out an error in one of our
codes.

\end{onehalfspace}

\bibliographystyle{plain}

\bibliography{arlp}

\begin{onehalfspace}

\section{Appendix}

\subsection{Proof of Lemma \ref{lemma:calculateVar2}}

\begin{proof}
Any feasible $y(\xi)$ satisfies
\begin{align*}
    y(\xi)_s
    &\ge \max\{\xi_s, 1 - \xi_s\} + y(\xi)_{s-1} \\
    &\ge \max\{\xi_s, 1 - \xi_s\} + \max\{\xi_{s-1}, 1 - \xi_{s-1}\} + y(\xi)_{s-2} \\
    &\ge \cdots \ge \sum_{i=1}^s \max\{\xi_i, 1 - \xi_i\}
\end{align*}
Hence, applying this inequality at an optimal $y(\cdot)$, it follows
that
\[
v_{\AR,2}^*
\ge \max \limits_{\xi \in \Xi_2} \sum_{i=1}^s \max\{\xi_i, 1 - \xi_i\}.
\]
Under the change of variables $\mu := 2 \xi - \mathbbm{1}_s$, we have 
\begin{align*}
v_{\AR,2}^*
&\ge
\max \limits_{\xi \in \Xi_2} \sum_{i=1}^s \max\{\xi_i, 1 - \xi_i\}
=
\max \limits_{\|\mu\| \le 1} \sum_{i=1}^s \tfrac12 \max\{1 + \mu_i, 1 - \mu_i\} \\
&=
\frac12 \max \limits_{\|\mu\| \le 1} \sum_{i=1}^s (1 + |\mu_i|)
=
\frac12 \left(s + \max \limits_{\|\mu\| \le 1} \|\mu\|_1 \right)
=
\tfrac12 (\sqrt{s} + s),
\end{align*}
where the last equality follows from the fact that the largest 1-norm
over the Euclidean unit ball is $\sqrt{s}$. Moreover, one can check that
the specific, sequentially defined mapping 
\begin{align*}  
y(\xi)_1 &:= \max\{\xi_1, 1 - \xi_1\} \\
y(\xi)_2 &:= \max\{\xi_2, 1 - \xi_2\} + y(\xi)_1 \\
&\vdots \\
y(\xi)_s &:= \max\{\xi_s, 1 - \xi_s\} + y(\xi)_{s-1}
\end{align*}
is feasible with objective value $\tfrac12(\sqrt{s} + s)$. So $v_{\AR,
2}^* \le \tfrac12(\sqrt{s} + s)$, and this completes the argument that
$v_{\AR, 2}^* = \tfrac12(\sqrt{s} + s)$.
\end{proof}

\subsection{Proof of Proposition \ref{prop:vib2Equalsvar2}}

The proof of Proposition \ref{prop:vib2Equalsvar2} requires
the following lemma. 

\begin{lemma} \label{lemma:existmu}
If a symmetric matrix $V$ is positive semidefinite on the null space of
the rectangular matrix $E$ (that is, $z \in \Null(E) \Rightarrow z^T V
z \ge 0$), then there exists $\mu > 0$ such that $\rho I + V + \mu E^TE
\succ 0$.
\end{lemma}

\begin{proof}
We prove the contrapositive. Suppose $\rho I + V + \mu E^TE$ is not
positive definite for all $\mu > 0$. In particular, there exists a
sequence of vectors $\{ z_\ell \}$ such that
\[
z_\ell^T(\rho I + V + \ell E^TE) z_\ell \le 0, \ \ \|z_\ell\| = 1.
\]
Since $\{ z_\ell \}$ is bounded, there exists a limit point $\bar z$ such
that
\[
z_\ell^T(\tfrac{1}{\ell} (\rho I + V) + E^TE)z_\ell \le 0
\ \ \Rightarrow \ \
\bar z^T E^T E \bar z = \| E z \|^2 \le 0
\ \ \Leftrightarrow \ \
\bar z \in \Null(E).
\]
Furthermore, 
\begin{align*}
z_\ell^T(\rho I + V) z_\ell \le -\ell z_\ell^TE^TEz_\ell = -\ell\|E z_\ell \|^2 \le 0
\ \ &\Rightarrow \ \
\bar z^T (\rho I + V) \bar z \le 0 \\
\ \ &\Leftrightarrow \ \
\bar z^T V \bar z \le -\rho \| \bar z\|^2 < 0.
\end{align*}
Thus, $V$ is not positive semidefinite on $\Null(E)$. 
\end{proof}

\begin{proof}[Proof of Proposition \ref{prop:vib2Equalsvar2}]
For fixed $\rho > 0$, let us construct the claimed feasible solution.
Set
\[
    \lambda = v^*_{\AR,2} = \tfrac12(\sqrt{s} + s), \quad
    \alpha = 0, \quad
    \tau = \tfrac14 \sqrt{s}, \quad
    S_{21} = 0,
\]
and
\[
    S_{22} = \frac{1}{2\sqrt{s}} \sum_{i=1}^{s}
    \left( f_{2i}f_{2i-1}^T + f_{2i-1}f_{2i}^T \right) \ge 0,
\]
where $f_{j}$ denotes the $j$-th standard basis vector in $\RR^m =
\RR^{2s}$. Note that clearly $\alpha \in \K_2$ and $\text{\rm Rows}(S_{21}) \in \K_2$. Also
forcing $v = \mu \mathbbm{1}_k$ for a single scalar variable $\mu$, where
$\mathbbm{1}_k$ is the all-ones vector of size $k = s+1$, the feasibility
constraints of (\ref{equ:temporal1}) simplify further to
\begin{equation} \label{equ:temporal1_simple}
\rho I +
\begin{pmatrix}
    \tfrac12 (s + \sqrt{s}) e_1e_1^T - \tfrac14 \sqrt{s} J &
    -\tfrac{1}{2}F^T \\
    -\tfrac{1}{2}F &
    -S_{22}
\end{pmatrix} 
+
\mu E^TE \succeq 0,
\end{equation}
where $e_1 \in \RR^k = \RR^{s+1}$ is the first standard basis vector. For
compactness, we write
\begin{equation} \label{equ:defnV}
    V := 
    \begin{pmatrix}
        \tfrac12 (s + \sqrt{s}) e_1e_1^T - \tfrac14 \sqrt{s} J &
        -\tfrac{1}{2}F^T \\
        -\tfrac{1}{2}F &
        -S_{22}
    \end{pmatrix}
\end{equation}
so that (\ref{equ:temporal1_simple}) reads $\rho I + V + \mu E^TE
\succeq 0$.

We next show that the matrix $V$ is positive semidefinite 
on $\Null(E)$. Recall that $E \in \RR^{n_2 \times (k+m)} = \RR^{s \times (3s+1)}$. For
notational convenience, we partition any $z \in \RR^{k+m}$ into $z = {u
\choose w}$ with $u \in \RR^k = \RR^{s+1}$ and $w \in \RR^m = \RR^{2s}$.
Then, from the definition of $E$, we have
\begin{align*}
z = {u \choose w} \in \Null(E)
\ \ &\Longleftrightarrow \ \
\left\{
    \begin{array}{l}
        w_1 + w_2 = w_3 + w_4 \\
        w_3 + w_4 = w_5 + w_6 \\
        \vdots \\
        w_{2s - 3} + w_{2s - 2} = w_{2s - 1} + w_{2s} \\
        w_{2s - 1} + w_{2s} = u_1
    \end{array}
\right\} \\
&\Longleftrightarrow \ \
w_{2i - 1} = u_1 - w_{2i} \ \ \forall \ i = 1, \ldots, s.
\end{align*}
So, taking into account the definition (\ref{equ:defnV}) of $V$,
\[
4 \, z^T Vz 
= 4 \, \textstyle{{u \choose w}}^T V {u \choose w}
= u^T \left( 2(s + \sqrt{s}) e_1e_1^T - \sqrt{s} J \right) u - 4 \, w^T F u - 4 \, w^T S_{22} w,
\]
which breaks into the three summands, and we will simplify each one by
one. First,
\begin{align*}
    u^T \left( 2(s + \sqrt{s}) e_1e_1^T - \sqrt{s} J \right) u
    &= 2(s + \sqrt{s}) u_1^2 - \sqrt{s} u_1^2 + \sqrt{s} \sum_{j=2}^{s+1} u_j^2 \\
    &= 2 \, s \, u_1^2 + \sqrt{s} \, u^T u.
\end{align*}
Second,
\begin{align*}
    -4w^T F u
    &= -4 \sum_{j = 1}^{2s} w_j [F u]_j
     = -4 \sum_{i = 1}^{s} \left( w_{2i - 1} [Fu]_{2i - 1} + w_{2i} [Fu]_{2i} \right) \\
    &= -2 \sum_{i = 1}^{s} \left( w_{2i - 1} (u_1 + u_{i+1}) + w_{2i} (u_1 - u_{i+1}) \right) \\
    &= -2 \sum_{i = 1}^{s} \left( (w_{2i - 1} + w_{2i}) u_1 + u_{i+1}(w_{2i-1} - w_{2i}) \right) \\
    &= -2 \sum_{i = 1}^{s} \left( u_1^2 + u_{i+1}(w_{2i-1} - w_{2i}) \right) \\
    &= -2 \, s \, u_1^2 - 2 \sum_{i = 1}^{s} u_{i+1}(w_{2i-1} - w_{2i}) \\
    &= -2 \, s \, u_1^2 + 2 \sum_{i = 1}^{s} u_{i+1}(w_{2i} - w_{2i-1})
     = -2 \, s \, u_1^2 + 2 \sum_{i = 1}^{s} u_{i+1}(2 w_{2i} - u_1).
\end{align*}
Finally, 
\begin{align*}
    - 4 w^T S_{22} w
    &= -4 w^T \left( \frac{1}{2\sqrt{s}} \sum_{i=1}^{s}
    \left( f_{2i}f_{2i-1}^T + f_{2i-1}f_{2i}^T \right) \right) w \\
    &= -\frac{4}{\sqrt{s}} \sum_{i=1}^{s} w_{2i - 1} w_{2i} 
     = -\frac{4}{\sqrt{s}} \sum_{i=1}^{s} (u_1 - w_{2i}) w_{2i}.
\end{align*}
Combining the three summands, we have as desired
\begin{align*}
    4 z^T Vz 
    &= \left( 2s \, u_1^2 + \sqrt{s} \, u^T u \right) +
    \left( -2s \, u_1^2 + 2 \sum_{i = 1}^{s} u_{i+1}(2w_{2i} - u_1) \right) +
    \left( -\frac{4}{\sqrt{s}} \sum_{i=1}^{s} (u_1 - w_{2i}) w_{2i} \right) \\
    &= \sqrt{s} \, u^T u + 2 \sum_{i = 1}^{s} u_{i+1}(2 w_{2i} - u_1)
    -\frac{4}{\sqrt{s}} \sum_{i=1}^{s} (u_1 - w_{2i}) w_{2i} \\
    &= \sum_{i = 1}^{s} \left(
    \frac{1}{\sqrt{s}}  \, u_1^2 + \sqrt{s} \, u_{i+1}^2 + 2 \, u_{i+1}(2 w_{2i} - u_1)
    -\frac{4}{\sqrt{s}}  (u_1 - w_{2i}) w_{2i}
    \right) \\
    &= \sum_{i = 1}^{s} \left(
    \frac{1}{\sqrt{s}}  \, u_1^2 - 2 \, u_1 \, u_{i+1}
    -\frac{4}{\sqrt{s}}  \, u_1 \, w_{2i} + \sqrt{s} \, u_{i+1}^2 + 4 \, u_{i+1} \, w_{2i}
    + \frac{4}{\sqrt{s}}  \, w_{2i}^2 
    \right) \\
    &= \sum_{i=1}^{s} \left(
        -  (s)^{-1/4}  \, u_1 + (s)^{1/4} \, u_{i+1} + 2(s)^{-1/4}  \, w_{2i}
    \right)^2 \\
    &\ge 0.
\end{align*}

Given that $\rho$, $V$, and $E$ are defined as above. By
Lemma \ref{lemma:existmu}, $\mu$ can be chosen so that 
(\ref{equ:temporal1_simple}) is indeed satisfied.
\end{proof}


\subsection{Proof of Proposition \ref{prop:dualoptimal}}

\begin{proof}
By construction, $Z$
is positive semidefinite, and one can argue in a straightforward manner
that
\[
Z_{11} = \Diag(1, \tfrac1s, \ldots, \tfrac1s), \quad
Z_{22} = \frac14 \left( I + \mathbbm{1}_m\mathbbm{1}_m^T - \sum_{i=1}^{s}(f_{2i}f_{2i-1}^T + f_{2i-1}f_{2i}^T) \right),
\]
and
\[
Z_{21} = \frac12 \begin{pmatrix}
1 &  \phantom{-}\tfrac{1}{\sqrt{s}} &  \phantom{-}0 & \cdots &  \phantom{-}0 \\
1 & -\tfrac{1}{\sqrt{s}}  &  \phantom{-}0 & \cdots &  \phantom{-}0 \\
1 &  \phantom{-}0 &  \phantom{-}\tfrac{1}{\sqrt{s}} & \cdots &  \phantom{-}0 \\
1 &  \phantom{-}0 & -\tfrac{1}{\sqrt{s}}  & \cdots &  \phantom{-}0 \\
\vdots &  \phantom{-}\vdots &  \phantom{-}\vdots & \ddots &  \phantom{-}\vdots \\
1 &  \phantom{-}0 &  \phantom{-}0 & \cdots &  \phantom{-}\tfrac{1}{\sqrt{s}}  \\
1 &  \phantom{-}0 &  \phantom{-}0 & \cdots & -\tfrac{1}{\sqrt{s}} 
\end{pmatrix}.
\]
Then $Z$ clearly satisfies $g_1g_1^T \bullet Z = 1$, $Z_{11}e_1 \in \K_2$,
$J \bullet Z_{11} \ge 0$, $Z_{22} \ge 0$, and $\text{\rm Rows}(Z_{21}) \in \K_2$.
Furthermore, the constraint $I \bullet Z \le r$ is easily satisfied for
sufficiently large $r$. To check the constraint $\diag(E Z E^T) = 0$, it
suffices to verify $EZ = 0$, which amounts to two equations. First,
\[
0 = E { 2e_1 \choose \mathbbm{1}_m } =  -2 \, d e_1^Te_1 + B^T \mathbbm{1}_m = -2d + 2d = 0,
\]
and second, for each $i=1, \ldots, s$,
\[
0 = E  { \tfrac{2}{\sqrt{s}}e_{i+1} \choose f_{2i-1} - f_{2i} } 
= - \frac{2}{\sqrt{s}} \, de_1^Te_{i+1} + B^T(f_{2i-1} - f_{2i}) 
= 0 + B^Tf_{2i-1} - B^Tf_{2i}
= 0. 
\]
So the proposed $Z$ is feasible. Finally, it is clear that the
corresponding objective value is $F \bullet Z_{21} = \tfrac12 (\sqrt{s}
+ s)$. So $Z$ is indeed optimal.
\end{proof}

\subsection{Proof of Proposition \ref{prop:case1}}

\begin{proof}
The inclusion $\Xi_1 \subseteq \Xi_2$ implies $\K_1 \subseteq \K_2$
and $\CP(\K_1 \times \RR_+^{2s}) \subseteq \CP(\K_2 \times \RR_+^{2s})$.
Hence, $\COP(\K_1 \times \RR_+^{2s}) \supseteq \COP(\K_2 \times \RR_+^{2s})$.
Moreover, it is not difficult to see that the construction of $\IB(\K_1 \times \RR_+^{2s})$
introduced at the end of Section \ref{sec:results} for
the polyhedral cone $\K_1$ satisfies $\IB(\K_1 \times \RR_+^{2s}) \supseteq
\IB(\K_2 \times \RR_+^{2s})$. Thus, we
conclude $v_{\IB,1}^* \le v_{\IB,2}^* = \tfrac12(\sqrt{s} + s)$.

We finally show $v_{\IB,1}^* \ge v_{\IB,2}^*$. Based on
the definition of $\K_1$ using the matrix $P$, similar to
(\ref{equ:temporal1_dual}) the corresponding dual problem is
\begin{equation} \label{equ:polyhedral}
\begin{array}{lll}
v_{\IB,1}^* = & \max & F \bullet Z_{21} \\
& \st & \diag (E Z E^T) = 0, \ I \bullet Z \le r \\
& & PZ_{11}e_1 \ge 0, \ PZ_{11}P^T \ge 0, \ Z_{22} \ge 0, \ P Z_{21}^T \ge 0 \\
& & Z \succeq 0, \ g_1g_1^T \bullet Z = 1.
\end{array}
\end{equation}
To complete the proof, we claim that the specific $Z$ detailed in the
previous subsection is also feasible for (\ref{equ:polyhedral}). It
remains to show that $PZ_{11}e_1 \ge 0$, $PZ_{11}P^T \ge 0$, and $PZ_{21}^T \ge 0$.

Recall that $Z_{11} = \Diag(1, \tfrac1s, \ldots, \tfrac1s)$ and every row of
$P$ has the form $(1, \pm 1, \ldots, \pm 1)$. Clearly, we have  $PZ_{11}e_1 \ge 0$.
Moreover, each entry of $PZ_{11}P^T$ can be expressed as ${1 \choose \alpha}^T Z_{11} {1 \choose \beta}$ for some
$\alpha, \beta \in \RR^s$ each of the form $(\pm 1, \ldots, \pm 1)$. We have
\[
    \textstyle{{1 \choose \alpha}}^T Z_{11} {1 \choose \beta} =
    1 + \frac1s \cdot \alpha^T \beta \ge 1 + \frac1s (-s) \ge 0.
\]
So indeed $PZ_{11}P^T \ge 0$. To check $PZ_{21}^T \ge 0$, recall
also that every column of $Z_{21}^T$ has the form $\tfrac12 (e_1 \pm
\tfrac{1}{\sqrt{s}} e_{i+1})$ for $i=1,\ldots,s$, where $e_{\bullet}$ is a
standard basis vector in $\RR^k = \RR^{s+1}$. Then each entry of $2 P Z_{21}^T$
can be expressed as
\[
    \textstyle{{1 \choose \alpha}}^T e_1 \pm \tfrac{1}{\sqrt{s}} {1 \choose \alpha}^T e_{i+1}
    \ge 1 - \tfrac{1}{\sqrt{s}} > 0.
\]
So $PZ_{21}^T \ge 0$, as desired.
\end{proof}

\end{onehalfspace}

\end{document}